\documentclass{article}
\usepackage[T1]{fontenc}
\usepackage[letterpaper]{geometry}

\usepackage{notation}

\usepackage{url}
\usepackage[colorlinks,citecolor=blue,urlcolor=blue,linkcolor=blue,linktocpage=true]{hyperref}
\usepackage{cleveref}

\usepackage[
  backend=biber,
  bibstyle=authoryear,
  citestyle=authoryear,
  maxbibnames=99,
  firstinits=true,
  sortcites=true,
  uniquename=false
]{biblatex}
\title{Squared-Norm Empirical Process in Banach Space\footnote{This note was first written in 2012 and a version has been available on the first author's website.  It is now posted on arXiv to provide a universal identifier for citations elsewhere.  It is not intended for publication in its current form, but a revision incorporating later work will be posted separately.}
\footnote{This research was supported by NSF grants DMS-0903120 and BCS-0941518.}}

\author{
Vincent Q.~Vu \\
Department of Statistics \\
The Ohio State University \\
Columbus, OH \\
\texttt{vqv@stat.osu.edu} \\
\and
Jing Lei \\
Department of Statistics \\
Carnegie Mellon University \\
Pittsburgh, PA \\
\texttt{jinglei@andrew.cmu.edu}
}

\date{}

\begin{document}

\maketitle
\begin{abstract}
This note extends a recent result of Mendelson on the supremum of a quadratic process 
to squared norms of functions taking values in a Banach space.  Our method of proof is a 
reduction by a symmetrization argument and observation about the subadditivity of the 
generic chaining functional.  We provide an application to the supremum of 
a linear process in the sample covariance matrix indexed by finite rank, 
positive definite matrices.
\end{abstract}

\section{Introduction}
\label{sec:introduction}
Let $F$ be a class of $\Real$-valued functions on the probability space $(\Omega, \prob)$ and 
$\{X_1,\ldots,X_n\}$ be independent, identically distributed random variables. 
Let $\norm{\cdot}_{\psi_\alpha}$ be the Orlicz $\psi_\alpha$ norm for $\alpha\ge 1$
and $\gamma_{\alpha}$ be Talagrand's generic chaining complexity of $F$ under
the $\psi_{\alpha}$ metric.
\citeauthor{Mendelson:2010} has proved the following theorem:
\begin{theorem}[{\textcite{Mendelson:2010}}]
  \label{thm:mendelson}
  If $F$ is a symmetric class of mean-zero functions on $(\Omega, \prob)$ then 
  there exists absolute constants $c_1$, $c_2$, and $c_3$ such that
  for any $t\ge c_1$, with probability at least $1-2\exp(-c_2 t^{2/5})$
  \begin{equation*}
    \sup_{f \in F}
    \Biggr\lvert
      \frac{1}{n} \sum_{i=1}^n f^2(X_i) - \E f^2(X_i)
    \Biggr\rvert
    \leq c_3 t
    \Biggr\{
      \frac{d_{\psi_1}(F) 
      \gamma_2(F, \psi_2)}{\sqrt{n}}
      +
      \frac{\gamma_2^2(F, \psi_2)}{n}
    \Biggr\}
    \,,
  \end{equation*}
  where $d_{\psi_1}(F) = \sup_{f\in F} \norm{f(X_1)}_{\psi_1}$.
\end{theorem}
The goal of this note is to extend the preceding theorem to a class $G$ 
of functions that take values in a Banach space $(B, \norm{\cdot})$.  
We wish to bound
\begin{equation}
  \label{eq:multivariate-sup}
   \sup_{g \in G}
    \Biggr\lvert
      \frac{1}{n} \sum_{i=1}^n \norm{g(X_i)}^2 - \E \norm{g(X_i)}^2
    \Biggr\rvert
\end{equation}
with a similar exponential tail probability bound.
 A special case is $g(X_1)=A^T X_1$, where $X_1$ is a random vector
in $\Real^p$ and $A\in \mathcal A$, a class of $p\times k$ matrices.
In this note we will assume that $G$ is countable to avoid unnecessary measurability issues.

When $(B,\norm{\cdot}) = (\Real, \lvert\cdot\rvert)$, this expectation reduces to the one addressed by 
\Cref{thm:mendelson}. Our goal is to bound the expectation in terms 
of complexity parameters of the class $G$.  We will use a symmetrization argument to 
put \eqref{eq:multivariate-sup} into a form suitable for \Cref{thm:mendelson}.
The next section contains an abstract technical result concerning subadditivity of 
the generic chaining $\gamma_\alpha$ functional \parencite{Talagrand} that enables 
our reduction.

\section{The Generic Chaining}
\label{sec:chaining}
\begin{definition}
  Given a metric space $(T,d)$ and $A$ a subset of $T$, define the \emph{diameter} 
  of $A$ to be
  \begin{equation}
    \Delta(A) \coloneqq \sup_{t,u \in A} d(t,u)\,.
  \end{equation}
\end{definition}
\begin{definition}
  Given a metric space $(T,d)$ and $\mathcal{A}$ a partition of $T$,
  for $t \in T$, define $\mathcal{A}(t)$ to be the unique element of $\mathcal{A}$ 
  that contains $t$.
\end{definition}

\begin{definition}[{\textcite{Talagrand}}]
  Given a metric space $(T,d)$,
  an \emph{admissible sequence} is an increasing sequence 
    $(\mathcal{A}_s)_{s \geq 0}$ of partitions of $T$ such that $\card{\mathcal{A}_0} = 1$ 
    and, for $s \geq 1$, $\card{\mathcal{A}_s} \leq 2^{2^s}$.
  For $\alpha \geq 1$, define the $\gamma_\alpha$ functional by 
  \begin{equation*}
    \gamma_\alpha(T, d) 
    \coloneqq \inf \sup_{t \in T} \sum_{s \geq 0} 2^{s/\alpha} \Delta(\mathcal{A}_s(t))
    \,,
  \end{equation*}
  where the infimum is taken over all admissible sequences.
\end{definition}

\begin{lemma}
  \label{lem:generic-chaining-subadditivity}
  If $(T_1, d)$ and $(T_2, d)$ are metric spaces, then
  \begin{equation*}
    \gamma_\alpha(T_1 \cup T_2, d) 
    \leq 
    3 \big[ \Delta(T_1 \cup T_2) + \gamma_\alpha(T_1, d) + \gamma_\alpha(T_2, d) \big]
    \,.
  \end{equation*}
  Moreover, if $T_1 \cap T_2 \neq \emptyset$, then
  \begin{equation*}
      \gamma_\alpha(T_1 \cup T_2, d) 
    \leq 
      9 \big[ \gamma_\alpha(T_1, d) + \gamma_\alpha(T_2, d) \big]
    \,.
  \end{equation*}
\end{lemma}
\begin{proof}
  Let $(\mathcal{A}_s)_{s \geq 0}$ be an admissible sequence for $T_1$ 
  such that
  \begin{equation}
    \label{eq:admissible-1}
    \sum_{s \geq 0} 2^{s/\alpha} \Delta(\mathcal{A}_s(t)) \leq 2 \gamma_\alpha(T_1, d) 
    \text{ for all $t \in T_1$},\,
  \end{equation}
  and $(\mathcal{B}_s)_{s \geq 0}$ an admissible sequence for $T_2$ 
  such that 
  \begin{equation}
    \label{eq:admissible-2}
    \sum_{s \geq 0} 2^{s/\alpha} \Delta(\mathcal{B}_s(t)) \leq 2 \gamma_\alpha(T_2, d) 
    \text{ for all $t \in T_2$}\,.
  \end{equation}
  We define partitions $\mathcal{C}_s$ of $T_1 \cup T_2$ as follows.  
  Let $\mathcal{C}_0 = \{ T_1 \cup T_2 \}$, 
  $\mathcal{C}_1 = \mathcal{C}_0$, 
  and, for $s \geq 2$, let $\mathcal{C}_s$ be the partition consisting of the 
  sets
  \begin{equation*}
    \label{eq:partition}
    \begin{gathered}
      A \cap (T_1 \setminus T_2) \text{ for $A \in \mathcal{A}_{s-2}$} \,, \\
      B \cap (T_2 \setminus T_1) \text{ for $B \in \mathcal{B}_{s-2}$} \,, \text{ and } \\
      A \cap B \text{ for $A \in \mathcal{A}_{s-2}$ and $B \in \mathcal{B}_{s-2}$} \,.
    \end{gathered}
  \end{equation*}
  It is straightforward to check that $(\mathcal{C}_s)_{s \geq 0}$ is an increasing 
  sequence of partitions of $T_1 \cup T_2$, 
  and that $\card{C_0} = \card{C_1} = 1$ and, for $s \geq 2$, 
  \begin{align*}
      \card{\mathcal{C}_s} 
    &\leq 
      \card{\mathcal{A}_{s-2}} + \card{\mathcal{B}_{s-2}}
      + \card{\mathcal{A}_{s-2}} \times \card{\mathcal{B}_{s-2}}
    \\
    &\leq
      2^{2^{s-2}} + 2^{2^{s-2}} + 2^{2^{s-1}}
    \leq
      2^{2^s}
    \,.
  \end{align*}
  Thus $(\mathcal{C}_s)_{s \geq 0}$ is an admissible sequence for $T_1 \cup T_2$.
  Let $t \in T_1 \cup T_2$.
  Since, for $s \geq 2$, every element of $\mathcal{C}_s$ is a subset of an element of 
  $\mathcal{A}_{s-2}$ and/or $\mathcal{B}_{s-2}$,
  \begin{align*}
      \sum_{s \geq 0} 2^{s/\alpha} \Delta(\mathcal{C}_s(t))
    &\leq
      3 \Delta(T_1 \cup T_2) 
      + \sum_{s \geq 2} 2^{s/\alpha} \Delta(\mathcal{C}_s(t))
    \\
    &\leq
      3 \Delta(T_1 \cup T_2)
      + \sum_{s \geq 0} 2^{s/\alpha} \Delta(\mathcal{A}_s(t))
      + \sum_{s \geq 0} 2^{s/\alpha} \Delta(\mathcal{B}_s(t))
    \\
    &\leq
      3
      \big[
        \Delta(T_1 \cup T_2) + \gamma_\alpha(T_1, d) + \gamma_\alpha(T_2, d)
      \big]
    \,,
  \end{align*}
  by \eqref{eq:admissible-1} and \eqref{eq:admissible-2}.
  Then 
  \begin{align*}
      \gamma_\alpha(T_1 \cup T_2, d)
    &\leq
      \sup_{t \in T_1 \cup T_2}
      \sum_{s \geq 0} 2^{s/\alpha} \Delta(\mathcal{C}_s(t))
    \\
    &\leq
      3 \big[
        \Delta(T_1 \cup T_2) + \gamma_\alpha(T_1, d) + \gamma_\alpha(T_2, d)
      \big]  
    \,.
  \end{align*}
  For the ``moreover'' part, if $u \in T_1 \cap T_2$, then by the triangle inequality
  \begin{align*}
      \Delta(T_1 \cup T_2)
    &= 
      \sup_{t_1, t_2 \in T_1 \cup T_2} d(t_1, t_2)
    \\
    &\leq 
      \sup_{t_1, t_2 \in T_1 \cup T_2} d(t_1, u) + d(t_2, u)
    \\
    &\leq
     2 \big[ \Delta(T_1) + \Delta(T_2) \big]
    \\
    &\leq
      2\big[ \gamma_2(T_1, d) + \gamma_2(T_2, d) \big]
    \,.
    \qedhere
  \end{align*}
\end{proof}

\section{Symmetrization}
\label{sec:symmetrization}
The following theorem is our main result.  Its proof is a symmetrization 
argument followed by an application of \Cref{thm:mendelson} to and 
\Cref{lem:generic-chaining-subadditivity}.
\begin{theorem}
  \label{thm:multivariate-mendelson}
  If $X_1, X_2, \ldots, X_n$ are independent, identically distributed random variables, 
  and $G$ is a symmetric class of functions taking values in a Banach space $(B, \norm{\cdot})$, 
  then there exists absolute constants $c_1$, $c_2$, and $c_3$ such that
  for all $t\ge c_1$, with probability at least $1-2\exp(-c_2t^{2/5})$,
  \begin{equation}
      \sup_{g \in G}
      \Biggr\lvert
        \frac{1}{n} \sum_{i=1}^n \norm{g(X_i)}^2 - \E \norm{g(X_i)}^2
      \Biggr\rvert
    \leq
      c_3t
      \Biggr\{
        \frac{
          \sup_{g \in G} \norm{\norm{g(X_1)}}_{\psi_1} 
          \gamma_2(G, d)
        }{\sqrt{n}}
      +
        \frac{
          \gamma_2^2(G, d)
        }{n}      
      \Biggr\}
      \,,    
  \end{equation}
  where $d$ is the metric on $G$ defined by 
  \begin{equation}
    d(g_1, g_2) = \bignorm{ \norm{g_1(X_1) - g_2(X_1)} }_{\psi_2}
    \,.
  \end{equation}
\end{theorem}
\begin{proof}
Let $\{\epsilon_1,\ldots,\epsilon_n\}$ 
be independent, identically distributed Rademacher random variables that are 
indepedent of $\{X_1,\ldots,X_n\}$. Define the function class
\begin{equation*}
  F = \{ f : f(x, \epsilon) = \epsilon \norm{g(x)} \text{ for some } g \in G \}
  \,,
\end{equation*}
and let 
\begin{equation*}
  F_0 = F \cup - F\,.
\end{equation*}
$F_0$ is a symmetric class of mean-zero functions by construction. 
So \Cref{thm:mendelson} implies that, for $t\ge c_1$, with
probability at least $1-2\exp(-c_2t^{2/5})$,
\begin{align*}
  &
    \sup_{g \in G}
    \Biggr\lvert
      \frac{1}{n} \sum_{i=1}^n \norm{g(X_i)}^2 - \E \norm{g(X_i)}^2
    \Biggr\rvert
  \\
    \quad =
  &
    \sup_{f \in F_0}
    \Biggr\lvert
      \frac{1}{n} \sum_{i=1}^n f^2(X_i, \epsilon_i) - \E f^2(X_i, \epsilon_i)
    \Biggr\rvert
  \\
    \quad \leq
  &
    c_3t\Biggr\{
      \frac{\sup_{g \in G} \bignorm{ \norm{g(X_1)} }_{\psi_1} \gamma_2(F_0, \psi_2)}{\sqrt{n}}
    +
      \frac{\gamma_2^2(F_0, \psi_2)}{n}
    \Biggr\}
  \,.
\end{align*}
The rest of the proof deals with showing that 
\begin{equation}
  \label{eq:main-contraction-bound}
  \gamma_2(F_0, \psi_2) \leq C \gamma_2(G, d)
  \,.
\end{equation}
Apply \Cref{lem:generic-chaining-subadditivity} to
\begin{align}
    \gamma_2(F_0, \psi_2)
  \notag
  &\leq
    3\big[
      \Delta(F \cup - F) + \gamma_2(F, \psi_2) + \gamma_2(-F, \psi_2)
    \big]
  \\
  \label{eq:subadditive-bound}
  &=
    3\big[
      \Delta(F \cup - F) + 2 \gamma_2(F, \psi_2)
    \big]
  \,.
\end{align}
For any $f_1,f_2 \in F$ with corresponding $g_1, g_2 \in G$, 
we have by the triangle inequality that
\begin{align*}
    \bignorm{f_1(X_1,\epsilon_1) - f_2(X_1,\epsilon_1) }_{\psi_2}
  &= 
    \bignorm{\norm{g_1(X_1)} - \norm{g_2(X_1)} }_{\psi_2} 
  \\
  &\leq 
    \norm{ \norm{g_1(X_1) - g_2(X_1)} }_{\psi_2}
  \\
  &\eqqcolon
    d(g_1, g_2)
  \,.
\end{align*}
Then by Theorem 1.3.6 of \textcite{Talagrand},
\begin{equation}
  \label{eq:contraction-bound}
  \gamma_2(F, \psi_2) \leq C \gamma_2(G, d)
\end{equation}
for an absolute constant $C > 0$.
Using the triangle inequality and the symmetry of $G$, 
\begin{align*}
    \Delta(F \cup - F)
  &=
    \max\Big\{ 
      \sup_{g_1,g_2 \in G}
      \bignorm{ \norm{ g_1(X_1) } - \norm{ g_2(X_1) } }_{\psi_2} 
      \,,\,
      \sup_{g_1,g_2 \in G}
      \bignorm{ \norm{ g_1(X_1) } + \norm{ g_2(X_1) } }_{\psi_2} 
    \Big\}
  \\
  &\leq
    \max\Big\{ 
      \sup_{g_1,g_2 \in G} d(g_1, g_2)
      \,,\,
      2 \sup_{g \in G} \bignorm{ \norm{ g(X_1) } }_{\psi_2} 
    \Big\}
  \\
  &=
    \max\Big\{ 
      \sup_{g_1,g_2 \in G} d(g_1, g_2)
      \,,\,
      \sup_{g \in G} \bignorm{ \norm{ g(X_1) - (-g(X_1)) } }_{\psi_2} 
    \Big\}
  \\
  &=
    \sup_{g_1, g_2 \in G} d(g_1, g_2)
  \\
  &\leq
    \gamma_2(G, d)
  \,.
\end{align*}
Substituting the preceding inequality and \eqref{eq:contraction-bound} into \eqref{eq:subadditive-bound} 
proves \eqref{eq:main-contraction-bound}.
\end{proof}

\section{Linear Transformations}
\label{sec:lineartransformations}
Let $\mathcal{A} \subset \Real^{p\times k}$
\begin{equation*}
  G = \{ g : g(x) = A^T x \text{ and } A \in \mathcal{A} \}
  \,.
\end{equation*}
Observe that
\begin{equation*}
  \frac{1}{n}
  \sum_{i=1}^n \norm{g(X_i)}_2^2 = \frac{1}{n} \sum_{i=1}^n \innerp{X_i X_i^T}{A A^T} \,.
\end{equation*}
To apply \Cref{thm:multivariate-mendelson} to this empirical process, we need to bound
\begin{equation}
  \label{eq:lt-diameter}
  \sup_{g \in G}
  \norm[\big]{ \norm{g(X_1)}_2 }_{\psi_1}
  =
  \sup_{A \in \mathcal{A}}
  \norm[\big]{ \norm{A^T X_1}_2 }_{\psi_1}
  \leq
  \sup_{A \in \mathcal{A}}
  \norm[\big]{ \norm{A^T X_1}_2 }_{\psi_2}
\end{equation}
and
\begin{equation}
  \label{eq:lt-metric}
  d(g_1, g_2)
  =
  \norm[\big]{
    \norm{g(X_1) - g_2(X_1)}_2
  }_{\psi_2}
  =
  \norm[\big]{
    \norm{(A_1 - A_2)^T X_1}_2
  }_{\psi_2}
  \,.
\end{equation}
The next lemma allows us to relate the $\psi_2$-norm on $G$ to the 
Frobenius norm on $\mathcal{A}$.
\begin{lemma}
  \label{lem:psi2-frobenius}
  Let $Z$ be a $p$-variate random vector and $A$ be a $p \times k$ matrix. 
  Then
  \begin{equation}
    \norm[\big]{ \norm{A^T Z}_2 }_{\psi_2} 
    \leq
    \norm{A}_F \norm{Z}_{\psi_2}\,,
  \end{equation}
  where $\norm{Z}_{\psi_2} \coloneqq \sup_{\norm{u}_2 \leq 1} \norm{\innerp{Z}{u}}_{\psi_2}$.
\end{lemma}
\begin{proof}
  Let $a_1,\ldots,a_k$ denote the columns of $A$. Then
  \begin{equation*}
    \norm[\big]{
      \norm{A^T Z}_2^2
    }_{\psi_1}
    \leq
    \sum_{i=1}^k \norm[\big]{\abs{\innerp{a_i}{Z}}^2}_{\psi_1}
    =
    \sum_{i=1}^k \norm[\big]{\innerp{a_i}{Z}}_{\psi_2}^2
    \leq
    \sum_{i=1}^k \norm{a_i}_2^2 \norm{Z}_{\psi_2}^2
    \,.
  \end{equation*}
  Thus,
  \begin{equation*}
    \norm[\big]{ \norm{A^T Z}_2 }_{\psi_2}^2
    =
    \norm[\big]{ \norm{A^T Z}_2^2 }_{\psi_1}
    \leq
    \norm{A}_F^2 \norm{Z}_{\psi_2}^2
    \,.
    \qedhere
  \end{equation*}
\end{proof}
Applying \Cref{lem:psi2-frobenius} to \cref{eq:lt-diameter,eq:lt-metric} yields 
\begin{equation*}
    \sup_{g \in G}
    \norm[\big]{ \norm{g(X_1)}_2 }_{\psi_1}
  \leq
    \sup_{A \in \mathcal{A}}
    \norm{X_1}_{\psi_2} \norm{A}_F
\end{equation*}
and
\begin{equation*}
  d(g_1,g_2) \leq \norm{X_1}_{\psi_2} \norm{A_1 - A_2}_F \,.
\end{equation*}
Since $\norm{A}_F = \norm{\vectorize(A)}_2$, the majorizing measure theorem 
\parencite[Theorem 2.1.1]{Talagrand} implies that
\begin{equation*}
  \gamma_2(G,d) \leq c \norm{X_1}_{\psi_2} \E \sup_{A \in \mathcal{A}} \innerp{A}{\mathcal{Z}} \,,
\end{equation*}
where $\mathcal{Z}$ is a $p \times k$ matrix with i.i.d. standard Gaussian entries 
and $c > 0$ is an absolute constant.  Thus, we have the following corollary of
\Cref{thm:multivariate-mendelson}.
\begin{corollary}
  Let $X_1,\ldots,X_n \in \Real^p$ be i.i.d. mean $0$ random vectors, $\Sigma = \E X_1 X_1^T$, 
  and $\sigma = \sup_{\norm{u}_2 = 1} \norm{\innerp{X_1}{u}}_{\psi_2}$, and 
  \begin{equation*}
    S_n = \frac{1}{n} \sum_{i=1}^n X_i X_i^T \,.
  \end{equation*}
  If $\mathcal{A} \subseteq \Real^{p \times k}$ is symmetric, then there exist positive constants
  $c_1$, $c_2$, and $c_3$ such that for all $t \geq c_1$, with probability at least $1 - 2 \exp(-c_2 t^{2/5})$,
  \begin{equation*}
    \sup_{A \in \mathcal{A}}
    \abs{\innerp{S_n - \Sigma}{AA^T}}
    \leq
    c_3 t \bigg\{
      \frac{\sigma^2 \sup_{A \in \mathcal{A}} \norm{A}_F}{\sqrt{n}}
      \big(
        \E \sup_{A \in \mathcal{A}} \innerp{\mathcal{Z}}{A}
      \big)
      +
      \frac{\sigma^2}{n}
      \big(
        \E \sup_{A \in \mathcal{A}} \innerp{\mathcal{Z}}{A}
      \big)^2
    \bigg\}
    \,,
  \end{equation*}
  where $\mathcal{Z}$ is a $p \times k$ matrix with i.i.d. standard Gaussian 
  entries.  Furthermore, we have under the same condition, 
  that 
  \begin{equation*}
    \E \sup_{A \in \mathcal{A}}
    \abs{\innerp{S_n - \Sigma}{AA^T}}
    \leq
    c \bigg\{
      \frac{\sigma^2 \sup_{A \in \mathcal{A}} \norm{A}_F}{\sqrt{n}}
      \big(
        \E \sup_{A \in \mathcal{A}} \innerp{\mathcal{Z}}{A}
      \big)
      +
      \frac{\sigma^2}{n}
      \big(
        \E \sup_{A \in \mathcal{A}} \innerp{\mathcal{Z}}{A}
      \big)^2
    \bigg\}
    \,,
  \end{equation*}
  with a universal constant $c$.
\end{corollary}

\printbibliography

\end{document}